\documentclass[11pt]{amsart}

\usepackage[T1]{fontenc}

\usepackage{amsmath,amssymb}
\usepackage{graphicx}
\usepackage{amscd}
\usepackage{mathptmx}
\usepackage{epsfig}
\numberwithin{equation}{section}

\newtheorem{theorem}{Theorem}[section]

\newtheorem{lemma}[theorem]{Lemma}

\newtheorem{proposition}[theorem]{Proposition}

\newcommand{\T}{{\mathcal T}}

\newcommand{\cC}{\mathcal{C}}
\newcommand{\curvskel}[1]{\mathcal{C}^{(#1)}\!}
\newcommand{\short}{\mathrm{sys}}

\def\mod{\mathrm{mod}}
\def\Ext{\mathrm{Ext}}

\newcommand{\ml}{\mathcal{ML}}

\makeatletter
 \let\c@theorem=\c@subsection
 \let\c@conjecture=\c@subsection
 \let\c@lemma=\c@subsection
 \let\c@proposition=\c@subsection
 \let\c@claim=\c@subsection
 \let\c@question=\c@subsection
 \let\c@criterion=\c@subsection
 \let\c@vfconj=\c@subsection
 \let\c@definition=\c@subsection
 \let\c@notation=\c@subsection
 \let\c@remark=\c@subsection
 \let\c@example=\c@subsection
 \let\c@equation=\c@subsection
 \let\c@figure=\c@subsection
 \let\c@wrapfigure=\c@subsection

\makeatother

\begin{document}
\title{Lipschitz constants to curve complexes}
\author{V. Gadre, 
E. Hironaka, 
R. P. Kent IV, 
and C. J. Leininger}
\thanks{Hironaka was partially supported by Simons Foundation grant \#209171, Kent by NSF grant DMS-1104871, and Leininger by NSF grant DMS-0905748. 
The authors thank the Park City Mathematics Institute, where this work was begun.}
\date{December 18, 2012}
\address{Department of Mathematics, Harvard University, One Oxford Street, Cambridge, MA 02138}
\email{vaibhav@math.harvard.edu}
\address{Department of Mathematics, 
Florida State University,
1017 Academic Way, 208 LOV,
Tallahassee, FL 32306}
\email{hironaka@math.fsu.edu}
\address{Department of Mathematics,
University of Wisconsin,
480 Lincoln Drive,
Madison, WI 53706}
\email{rkent@math.wisc.edu}
\address{Department of Mathematics, 
University of Illinois at Urbana--Champaign,
1409 W. Green St.
Urbana, IL 61801}
\email{clein@math.uiuc.edu}

\maketitle


\begin{abstract}
We determine the asymptotic behavior of the optimal Lipschitz constant for the systole map from Teichm\"uller space to the curve complex. 
\end{abstract}

\maketitle

\section{Introduction}
Let $S = S_g$ be a closed surface of genus $g \geqslant 2$.
We equip the Teichm\"uller space  $\T(S)$ of $S$ with the Teichm\"uller metric, and equip the $1$--skeleton $\curvskel{1}(S)$ of the complex of curves $\cC(S)$ with its usual path metric $d_\cC$.  

In \cite{Mas-Min}, Masur and Minsky study the {\em systole map}
\[ 
	\short: \T(S) \to \curvskel{1}(S),
\]
which assigns a hyperbolic metric one of its shortest curves, called a {\em systole}.  
They prove that $\short$ is {\em $(K,C)$--coarsely Lipschitz} for $K, C > 0$, meaning that, for all $X$ and $Y$ in $\T(S)$
\[ 
	d_\cC(\short(X),\short(Y)) \leqslant K d_T(X,Y) + C.
\]
This is the starting point of their proof that $\curvskel{1}(S)$ is $\delta$--hyperbolic.
(The constant $\delta$ has recently been shown to be independent of $g$ by Aougab \cite{Aou}, Bowditch \cite{Bowditch.Curve}, and Clay, Rafi, and Schleimer \cite{Clay.Rafi.Schleimer}.)

In this paper we consider the {\em optimal Lipschitz constant}
\[ 
	\kappa_g = \inf \{ K \geqslant 0 \mid \short \mbox{ is } (K,C)\mbox{--coarsely Lipschitz for some } C > 0\}.
\]
We write $F(g) \asymp H(g)$ to mean that $F(g)/H(g)$ is bounded above and below by two positive constants, and prove the following theorem.
\begin{theorem} \label{T:main}
As $g \to \infty$ we have
\[ 
	\kappa_g \asymp \frac{1}{\log(g)}.
\]
\end{theorem}
\noindent This is a sharp version of the closed case of Theorem 1.4 of \cite{Aou}, which provides a Lipschitz constant that is independent of $\chi(S)$.
An analogous result holds when hyperbolic length is replaced with extremal length, see Proposition \ref{Extremal.Length.Corollary}.

The upper bound on $\kappa_g$ is established by a careful version of Masur and Minsky's proof that $\short$ is coarsely Lipschitz.  
To establish the lower bound, we construct a sequence of pseudo-Anosov mapping classes whose translation lengths on $\T(S)$ and $\curvskel{1}(S)$ behave like $\log(g)/g$ and $1/g$, respectively.

\section{A Lipschitz constant.}

Given the isotopy class $[f:S \to X]$ of a marked hyperbolic surface and the homotopy class of a curve $\alpha$, we write $\ell_X(\alpha)$ for the hyperbolic length of $\alpha$ in $[f:S \to X]$.  
Let $\short(X)$ denote the set of $\alpha$ in $ \curvskel{0}(S)$ for which $\ell_X(\alpha)$ is minimal.  
If $\alpha$, $\beta$ are in $\short(X)$, then the geometric intersection number $i(\alpha,\beta)$ is at most $1$, and so the diameter of $\short(X)$ in $\curvskel{1}(S)$ is at most $2$.  
We abuse notation and view $\short$ as a map from $\T(S)$ to $\curvskel{1}(S)$, although the image of $X$ is actually a subset of diameter at most $2$.
One may obtain a {\em bona fide} map via the Axiom of Choice.

Given a hyperbolic surface $X$ and a geodesic $\alpha$ on $X$, a {\em collar neighborhood of width $r$} about $\alpha$ is an $r$--neighborhood whose interior is homeomorphic to an open annulus. 
We have the following lemma.

\begin{lemma} \label{L:collar on short}
Given a closed hyperbolic surface $X$, if $\alpha$ lies in $\short (X)$, then there is a collar neighborhood of $\alpha$ of width greater than $\ell_X(\alpha)/2$.
\end{lemma}
\begin{proof}
Consider a maximal--width collar neighborhood $N_{w/2}(\alpha)$ of width $w$.  This has a self--tangency on its boundary.  From this one can construct a curve $\gamma$ that runs a distance $w/2$ from one of the points of tangency to $\alpha$, then at most half--way around $\alpha$ a distance at most $\ell_X(\alpha)/2$, and then a distance $w/2$ to the second point of tangency. Since $\alpha$ is a systole, we have
\[ 
	\ell_X(\alpha) \leqslant \ell_X(\gamma) < w + \ell_X(\alpha)/2. 
\]
So $w > \ell_X(\alpha)/2$ as required.
\end{proof}

Recall that a pair of isotopy classes of curves {\em fills} $S$ if, whenever the curves are realized transversally, the complement of their union is a set of topological disks.

\begin{lemma} \label{L:intersection number bound}
Given $\alpha$ and $\beta$ in $\curvskel{0}(S)$ that fill the surface $S$, we have
\[ 
	i(\alpha,\beta) \geqslant  2g-1. 
\]
\end{lemma}
\begin{proof}
The union $\alpha \cup \beta$ is a graph on $S$ with $i(\alpha,\beta)$ vertices and $2 i(\alpha,\beta)$ edges. 
The complement is a union of $F \geqslant 1$ disks.  
Therefore
\[ 
	2g - 2 = -\chi(S) = - i(\alpha,\beta) + 2 i(\alpha,\beta)  - F = i(\alpha,\beta) - F \leqslant i(\alpha,\beta) - 1.
\]
So $i(\alpha,\beta) \geqslant 2g-1$ as required.
\end{proof}

We need Wolpert's inequality \cite{Wolpert} describing change in lengths in terms of the Teichm\"uller distance.

\begin{lemma}[Wolpert, Lemma 3.1 of \cite{Wolpert}] \label{L:Wolpert}
Given $X,Y \in \T(S)$ and a curve $\alpha$ on $S$ we have
\[ 
\ell_Y(\alpha) \leqslant e^{d_\T(X,Y)} \ell_X(\alpha). 
\]
\qed
\end{lemma}

Our upper bound on $\kappa_g$ now follows from the following proposition.
\begin{proposition} \label{P:upper bound}
For $g \geqslant 2$ and all $X,Y \in \T(S_g)$ we have
\[ 
	d_\cC(\short(X),\short(Y)) \leqslant \frac{2}{\log(g-\frac{1}{2})} d_\T(X,Y) + 2.
\]
\end{proposition}

\begin{lemma} \label{E:small distance lipschitz} If $d_\T(X,Y) \leqslant \log\left(g- 1/2\right)$, then $d_\cC(\short(X),\short(Y)) \leqslant 2$. 
\end{lemma}
\begin{proof}
Suppose that $d_\T(X,Y) \leqslant \log(g-1/2)$.  
Write $\alpha = \short(X)$ and $\beta = \short(Y)$, and, without loss of generality, assume that
\[ 
	\ell_X(\alpha) \leqslant \ell_Y(\beta). 
\]
According to Lemma \ref{L:collar on short}, we have
\[ 
	\frac{i(\alpha,\beta) \ell_Y(\beta)}{2} < \ell_Y(\alpha).
\]
On the other hand, Lemma \ref{L:Wolpert} implies that
\[ 
	\ell_Y(\alpha) \leqslant e^{\log(g-1/2)} \ell_X(\alpha) = (g-1/2)\ell_X(\alpha) = \frac{(2g-1)}{2} \ell_X(\alpha).
\]
Combining these two inequalities yields
\[ 
i(\alpha,\beta) < \frac{2 \ell_Y(\alpha)}{\ell_Y(\beta)} \leqslant \frac{(2g-1) \ell_X(\alpha)}{\ell_Y(\beta)} \leqslant 2g-1.
\]
By Lemma \ref{L:intersection number bound}, $\alpha$ and $\beta$ cannot fill the surface $S$, and hence
\[ 
d_\cC(\short(X),\short(Y)) = d_\cC(\alpha,\beta) \leqslant 2.
\]
This proves the claim.
\end{proof}

\begin{proof}[Proof of Proposition~\ref{P:upper bound}]
Now, given any two points $X$ and $Y$ in $\T(S)$, let $n$ be the nonnegative integer such that
\[ 
n \log(g-1/2) \leqslant d_\T(X,Y) < (n+1)\log(g-1/2).
\]
Let $X = X_0,\ldots,X_{n+1} = Y$ be a chain in $\T(S)$ with
\[ 
d_\T(X_{k-1},X_k) \leqslant \log(g-1/2)
\]
for each $1 \leqslant k \leqslant n+1$. 
By the triangle inequality and (\ref{E:small distance lipschitz}), we have
\begin{align*}
d_\cC(\short(X),\short(Y)) & \leqslant  \sum_{k=1}^{n+1} d_\cC(\short(X_{k-1}),\short(X_k))\\
 & \leqslant  2 (n+1) \\
 &\leqslant \frac{2}{\log(g-1/2)} d_\T(X,Y) + 2
\end{align*}
as required.
\end{proof}

\section{pseudo-Anosov maps}

Given a pseudo-Anosov homeomorphism $f:S \to S$, we let $\lambda(f)$ denote the dilatation of $f$.  
We recall a few facts about pseudo-Anosov homeomorphisms, and refer the reader to the listed references for more detailed discussions.

\subsection{Asymptotic translation length.}  Given a homeomorphism $f:S \to S$, the asymptotic translation length of $f$ on $\curvskel{1}(S)$ is defined by
\[
\ell_{\cC}(f)= \liminf _{j\rightarrow \infty} \frac{d_\cC(\alpha, f^j(\alpha))}{j},
\]
where $\alpha$ is any simple closed curve.  
This is easily seen to be independent of $\alpha$.
When $f$ is pseudo-Anosov, Masur and Minsky proved $f$ has a quasi-invariant geodesic axis, and so this limit infimum is in fact a limit.  
Moreover, there is a $C > 0$ depending only on the genus of $S$ such that $\ell_\cC(f) \geqslant C$, see \cite{Mas-Min} or Corollary of 1.5 \cite{Bow}. It follows from the definition that $\ell_{\cC}(f^k) = k \ell_{\cC}(f)$. 

One can similarly define the asymptotic translation length of $f:S \to S$ acting on $\T(S)$.  
A pseudo-Anosov $f$ has an axis in $\T(S)$ (see \cite{Bers}), and the asymptotic translation length is just the translation length $\ell_\T(f)$.  
In fact, Bers' proof of Thurston's classification theorem shows that
\[ 
	\ell_\T(f) = \log(\lambda(f)).
\]

The following lemma allows us to use asymptotic translation lengths to bound optimal Lipschitz constants.

\begin{lemma} \label{L:comparing asymptotics}
For any pseudo-Anosov $f:S_g \to S_g$ we have
\[ 
	\kappa_g \geqslant \frac{\ell_\cC(f)}{\log(\lambda(f))}.
\]
\end{lemma}
\begin{proof}
If $K,C > 0$ are such that $\short$ is $(K,C)$--coarsely Lipschitz, then, for any $X$ in $\T(S)$, we have
\begin{align*}
\frac{\ell_\cC(f)}{\log(\lambda(f))}
	& =  \lim_{j \to \infty}\frac{d_\cC(\short(X),f^j(\short(X)))}{d_\T(X,f^j(X))}\\
	& =   \lim_{j \to \infty}\frac{d_\cC(\short(X),\short(f^j(X)))}{d_\T(X,f^j(X))}\\
	& \leqslant  \lim_{j \to \infty} \frac{Kd_\T(X,f^j(X))+C}{d_\T(X,f^j(X))}\\
	& \leqslant K. 
\end{align*}
Since $\kappa_g$ is the infimum of these $K$, the lemma is proven.
\end{proof}

\subsection{Invariant train tracks for pseudo-Anosov maps.}

For more on train tracks, we refer the reader to \cite{Pen-Har}, whose notation we adopt.

Given a pseudo-Anosov map $f:S \to S$, let $\tau$ denote an invariant train track. 
So $\tau$ carries $f(\tau)$, written $f(\tau) \prec \tau$, and a carrying map sends vertices of $f(\tau)$ to vertices of $\tau$.
Let $P_\tau$ denote the polyhedron of measures on $\tau$, viewed either as the space of weights on the branches $B$ of $\tau$ satisfying the switch conditions (a cone in $\mathbb R_{\geqslant 0}^B$), or a subset of the space $\ml(S)$ of measured laminations on $S$.

Although the carrying map is not unique, $f$ induces a canonical linear inclusion $f_*:P_\tau \subset P_\tau$.  
There is a unique eigenray in $P_\tau$ spanned by the stable lamination, and the corresponding eigenvalue is the dilatation $\lambda(f)$.  
In fact, this is the unique eigenray in all of $\mathbb R_{\geqslant 0}^B$ with eigenvalue greater than one.

\begin{theorem}
If $\tau$ is an invariant train track for a pseudo-Anosov homeomorphism $f:S \to S$ with transition matrix $A$, then $\lambda(f)$ is the spectral radius of $A$.
\qed
\end{theorem}

The dilatation $\lambda(f)$ is also the spectral radius of the matrix that defines the map
\[
	\mathbb R_{\geqslant 0}^B  \to \mathbb R_{\geqslant 0}^B,
\]
induced by $f$.
Furthermore, given any $f$--invariant subspace $V$ of $P_\tau$, the dilatation is the spectral radius of the matrix (with respect to any basis) defining the map $V \to V$ induced by $f$.  
If the matrix is a nonnegative integral matrix $A$, there is an associated directed graph, a {\em digraph}, with vertices the basis vectors, and $A_{ij}$ edges from the $i^{\mathrm{th}}$ basis vector to the $j^{\mathrm{th}}$ basis vector.  

\subsection{Basic Nesting Lemma and lower bound for asymptotic translation length.} 
A maximal train track $\tau$ is {\em recurrent} if there is some $\mu$ in $P_\tau$ that has positive weights on every branch. The set of such $\mu$ will be denoted  $\text{int}(P_\tau)$. A maximal train track $\tau$ is {\em transversely recurrent} if every branch intersects some 
closed curve that intersects $\tau$ efficiently. A train track that is both recurrent and transversely recurrent is called birecurrent.

For a maximal train track $\tau$, Masur and Minsky observed that if $\alpha$ is a curve in $\text{int}(P_\tau)$ and a curve $\beta$ is disjoint from $\alpha$, then $\beta$ is in $P_\tau$, see  Observation 4.1 of \cite{Mas-Min}. 
From this they deduce the following proposition.

\begin{proposition} \label{P:nested train track}
If $\tau$ is a maximal birecurrent invariant train track for a pseudo-Anosov $f:S \to S$ and $r \geqslant 1$ is such that $f^r(P_\tau) \subset \text{int}(P_\tau)$, then \[
\ell_\cC(f) \geqslant 1/r.
\]
\qed
\end{proposition}
We call an $r$ satisfying the conditions of Proposition~\ref{P:nested train track}
a {\em mixing number} for $f$ and $\tau$. 
In the next section, we construct a family of pseudo-Anosov maps $\phi_g: S_g \to S_g$ and maximal birecurrent invariant train tracks $\tau_g$ with mixing numbers $2g-1$. 

\section{Lower bound on $\kappa_g$.}

We build a family of pseudo-Anosov maps $\{ \phi_g:S_g \to S_g \}$ for which the asymptotic translation lengths on $\T(S_g)$ are on the order of $\log g/ g$ while the asymptotic translation lengths on $\curvskel{1}(S_g)$ are bounded below by a linear function of $g$.  
The lower bound on $\kappa_g$ in Theorem \ref{T:main} follows from this and Lemma \ref{L:comparing asymptotics}.
Our construction is similar to Penner's \cite{Pen}, but the asymptotic behavior is different.

\begin{figure}[htb]
\begin{center}
\includegraphics[height=12cm]{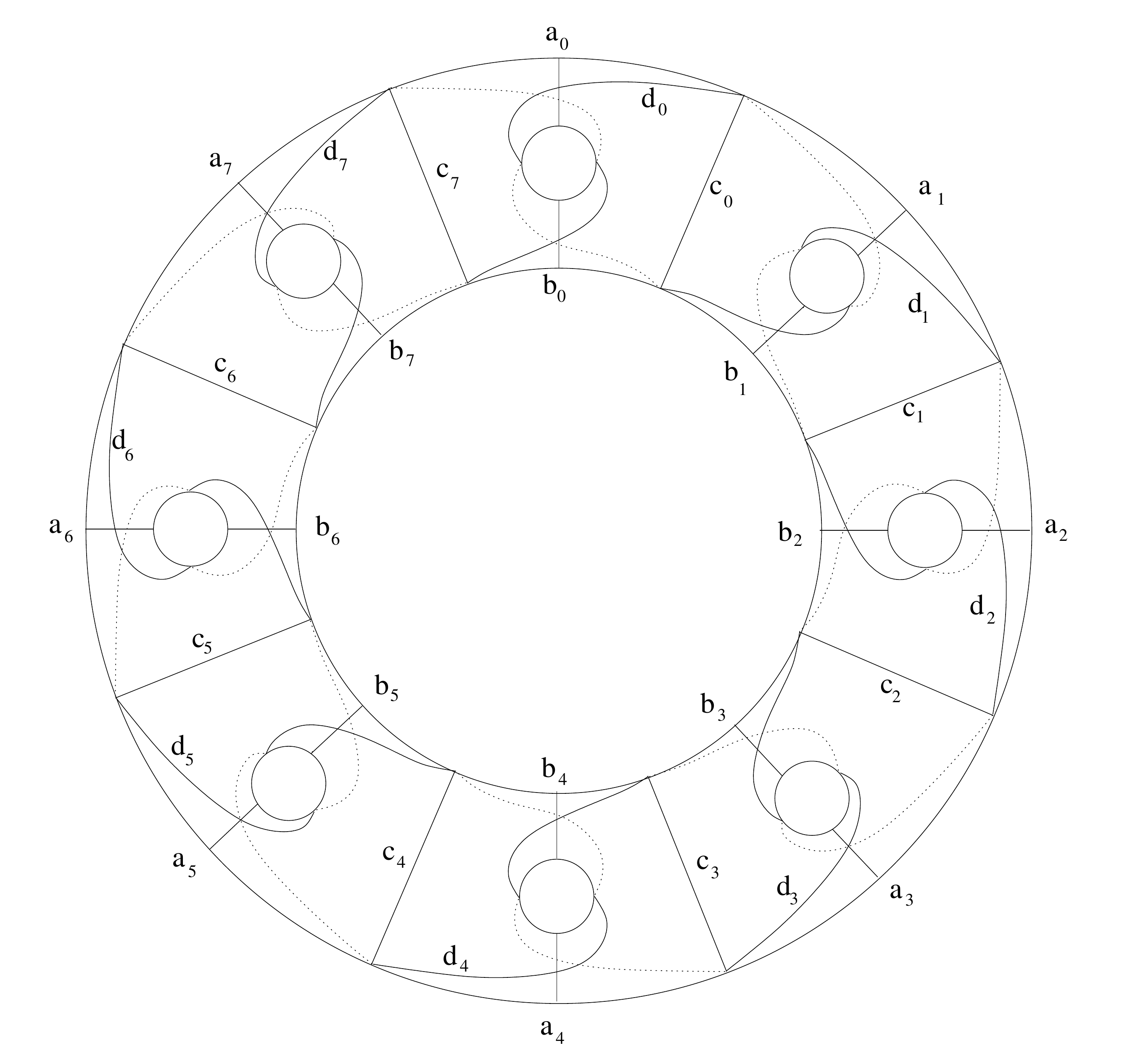} \label{F:genus9lip1}\caption{The pseudo-Anosov $\phi_9$}
\end{center}
\end{figure}

Let $g \geqslant 4$ and consider the genus $g$ surface $S = S_g$ with curves
\[ 
\Omega = \Omega_g = \{ a_0,\ldots,a_{g-2}, b_0,\ldots, b_{g-2}, c_0,\ldots,c_{g-2}, d_0,\ldots,d_{g-2} \}
\]
as indicated in Figure \ref{F:genus9lip1} when $g = 9$.  
For a curve $x$ in $\Omega$, let $T_x$ be the left--handed Dehn twist in $x$.
Let $\rho= \rho_g$ be the symmetry of order $g-1$ obtained by rotating $S_g$ clockwise by $2\pi/(g-1)$, and let
\[ 
	\phi = \phi_g = \rho_g \circ T_{a_0} \circ T_{b_1} \circ T_{c_0} \circ T_{d_0}^{-1} .
\]

Observe that the only nonzero intersection numbers among curves in $\Omega$ are
\[ 
i(d_j,a_j) = i(d_j,a_{j+1}) = i(d_j,b_j) = i(d_j,b_{j+1}) = 1 \mbox{\ and \ } i(d_j,c_j) = 2
\]
for $j \in \{0,\ldots,g-2\}$, where indices are taken modulo $g-1$.  
Smoothing intersection points as indicated in Figure \ref{F:smoothing}, we produce a maximal train track $\tau = \tau_g$.  
Each of the curves in $\Omega$ is carried by $\tau$, proving that $\tau$ is recurrent, and these curves are elements of $P_\tau$.  
Moreover, each of the curves can be pushed off $\tau$ to meet it efficiently, proving that $\tau$ is transversely recurrent.  
Let $P_\Omega \subset P_\tau$ be the subspace of measures carried by $\tau$ that lie in the span of $\Omega$.  
Because no two curves of $\Omega$ put nonzero weights on the same set of branches, the set $\Omega$ is a basis for $P_\Omega$.

\begin{figure}[htb]
\begin{center}
\includegraphics[height=3cm]{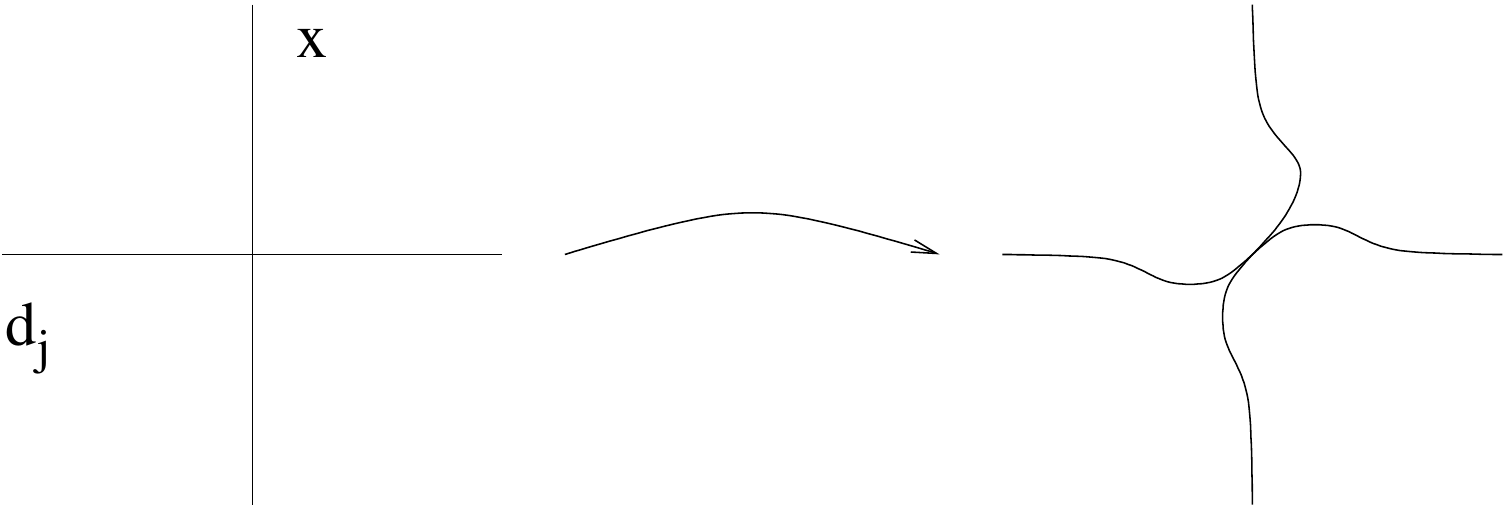} 
\caption{Smoothing the intersection points.  Here $x$ is some $a_i, b_i,$ or $c_i$.}\label{F:smoothing}
\end{center}
\end{figure}

Since $\Omega$ is $\rho$--invariant, we may assume that $\tau$ is.  
Furthermore, one has that $T_{a_j}(\tau)$, $T_{b_j}(\tau)$, $T_{c_j}(\tau)$, and $T_{d_j}^{-1}(\tau)$ are carried by $\tau$ for any $j$, as in \cite{PenConst}.  
In fact, we have $f(P_\Omega) \subset P_\Omega$ for any $f$ in 
$\{\rho, T_{d_j}^{-1}, T_{a_j}, T_{b_j}, T_{c_j} \ | \ 0 \leqslant j \leqslant g-1\}$.  
It follows that $\phi(P_\Omega) \subset P_\Omega$ and, as in \cite{Pen}, $\phi$ is pseudo-Anosov.  
Let $A$ denote the matrix for the action of $\phi$ on $P_\Omega$ in terms of the basis $\Omega$.  
This is a Perron--Frobenius matrix whose associated digraph $G_g$ is shown in Figure \ref{F:genus9lip2} in the case $g = 9$.  
The vertices are labeled by the corresponding elements of $\Omega$, and multiple edges are represented by an edge labeled with the multiplicity.
An important feature is that $G$ has exactly one self--loop, at the vertex $a_1$.

\begin{figure}[htb]
\begin{center}
\includegraphics[height=11.5cm]{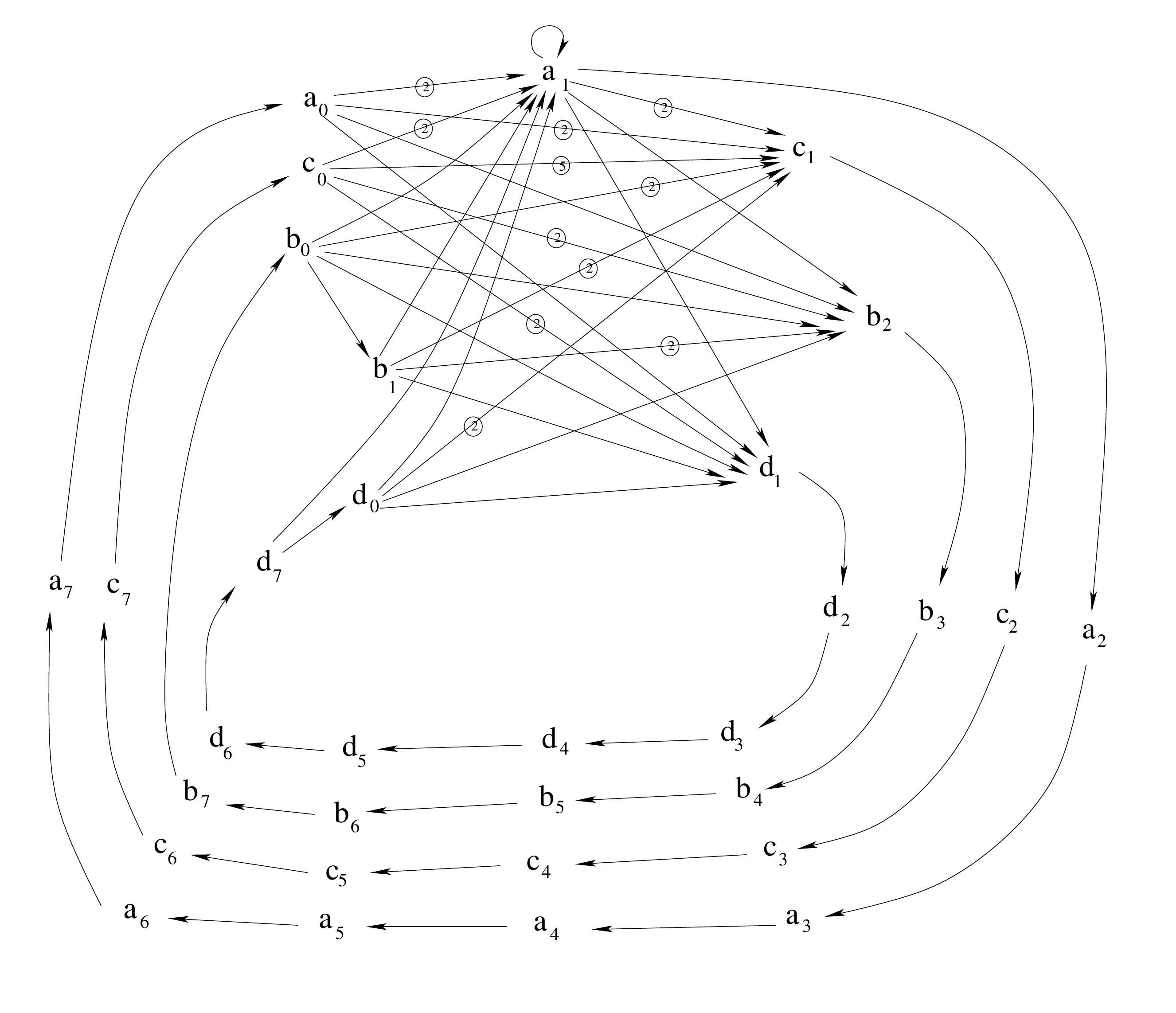}
\caption{The digraph $G_9$.}
\label{F:genus9lip2}
\end{center}
\end{figure}

First we bound the translation length on $\curvskel{1}(S)$ from below.

\begin{proposition} \label{P:cc calc}
For every $g \geqslant 4$, 
\[
\ell_\cC(\phi_g) \geqslant \frac{1}{2g-1}.
\]
\end{proposition}
\begin{proof}
By Proposition \ref{P:nested train track}, it is enough to show that $r = 2g-1$ is a mixing number for $\phi$ and $\tau$. 
We show this in two steps. 

We first show that, for any $\mu \in P_\tau$, there is an $s \leqslant g$ so that $\phi^s(\mu) = t a_1 + \mu'$ for some $t > 0$ and $\mu' \in P_\tau$.
Observe that $\mu$ has positive intersection number with some curve $a_j$ or $d_j$.  
Indeed, if we push all of the $a_j$ and $d_j$ off of $\tau$ in both directions so as to meet it efficiently, then the union of these curves intersects every branch.  Next, set $s_0 = g-1 - j$, so that $1 \leqslant s_0 \leqslant g-1$.  Then $\mu_{s_0} = \phi^{s_0}(\mu)$ has positive intersection number with either $a_0$ or $d_0$.  From this we have
\begin{align*}
 T_{a_0} T_{d_0}^{-1}(\mu_{s_0}) 
 	& =   \mu_{s_0} + i(\mu_{s_0},d_0) d_0 + i(\mu_{s_0} + i(\mu_{s_0},d_0)d_0, a_0)a_0\\
 	& =  \mu_{s_0} + i(\mu_{s_0},d_0)d_0 + \left( i(\mu_{s_0},a_0) + i(\mu_{s_0},d_0)i(d_0,a_0) \right) a_0\\
 	& =  \mu_{s_0} + i(\mu_{s_0},d_0)d_0 + \left( i(\mu_{s_0},a_0) + i(\mu_{s_0},d_0) \right) a_0.
 \end{align*}
Applying $\rho T_{b_1}T_{c_0}$ to this is the same as applying $\phi$ to $\mu_{s_0}$ since $T_{a_0}$ commutes with $T_{b_1}T_{c_0}$.  Therefore
\[ \phi^{s_0+ 1}(\mu) = \phi(\mu_{s_0}) =  t a_1 + \mu '\] where
\begin{align*}
s &= s_0 + 1,\\
t &= i(\mu_{s_0},a_0) + i(\mu_{s_0},d_0) >0, \quad \mbox{and}\\
\mu' &= \rho T_{b_1}T_{c_0}(\mu_{s_0} + i(\mu_{s_0},d_0)d_0 ) \in P_\tau.\\
\end{align*}

The second step is to show that, for any $k \geqslant g-1$, we have $\phi^k(a_1) \in \mathrm{int}(P_\tau)$.   
This follows from the fact that, for any $k \geqslant g-1$, there is a path of length $k$ from $a_1$ to any other vertex $x \in \Omega$, see Figure \ref{F:genus9lip2}.

From these two steps,  we have
\begin{align*} 
\phi^{2g-1}(\mu)& = \phi^{2g-1-s}(\phi^s(\mu))\\
& = \phi^{2g-1-s}(t a_1 + \mu') \\
&= t \phi^{2g-1-s}(a_1) + \phi^{2g-1-s}(\mu').
\end{align*}
The iterate $s$ from step one satisfies $2g-1-s \geqslant g-1$. 
By step two, we know that the right--hand side lies in $\mathrm{int}(P_\tau) + P_\tau \subset \mathrm{int}(P_\tau)$.  
It follows that $\phi^{2g-1}(P_\tau) \subset \mathrm{int}(P_\tau)$ and so $2g-1$ is a mixing number for $\phi$ and $\tau$. 
\end{proof}

\subsection{Bounds on dilatations.} 

\begin{lemma}\label{LogGrowth-lem} For $g > 4$, the mapping classes $\phi_g$ satisfy
\[
\frac{\log(4g-4)}{2g-2} \leqslant \log(\lambda(\phi_g)) \leqslant \frac{\log(10g-21)}{g-2}.
\]
\end{lemma}
\begin{proof}
The lower bound holds for any Perron--Frobenius digraph with a self--loop, thanks to work of Tsai (Proposition 2.4 of \cite{Tsa}), and so we prove only the upper bound.

For any $j \leqslant g-2$, inspection reveals that the number of directed edge--paths in $G_g$ of length $j$ emanating from each of
\[ 
a_0, \ a_1, \ b_0, \ b_1, \ c_0, \ d_{g-2}, \ \mathrm{and}\ d_0
\]
to be
\[ 
(10j-6), \
5j, \
(10j-1), \
5j, \
(10j-6), \
(10j-11), \ \mathrm{and} \ 
(5j-1), \
\]
respectively---see Figure \ref{F:genus9lip2}.   For any other vertex $v$ of $G_g$, there is a unique edge--path starting at $v$ and ending at one of the vertices listed above, and every shorter edge--path is an initial segment of this one.  
It follows that the number of edge--paths of length $g-2$ starting at any vertex is maximized at one of the vertices listed above, and is hence at most $10g-21$.  

Let $A_g$ be the incidence matrix of $G_g$.  The maximum row sum of $A_g^{g-2}$ is precisely the maximum number of edge--paths starting at any vertex, and is hence at most $10g-21$.   
But the maximum row sum of a Perron--Frobenius matrix is an upper bound for its spectral radius.
Applying this to $A_g^{g-2}$ we have
\[ 
\log(\lambda(\phi_g)) = \frac{\log(\lambda(\phi_g)^{g-2})}{g-2} = \frac{\log(\lambda(\phi_g^{g-2}))}{g-2} \leqslant \frac{\log(10g-21)}{g-2}. \qedhere
\]
\end{proof}

Alternatively, one may calculate the characteristic polynomial $P_{G_g}(x)$ of $G_g$ by observing that the mapping classes $\phi_g$ are the monodromies of fibrations of a single $3$--manifold.  
In fact, all of the fibers lie in a single cone on a fibered face of the Thurston norm ball, and one can use the Teichm\"uller polynomial to calculate the $P_{G_g}(x)$ by specializing a single polynomial.  
See \cite{McMullen}.
The polynomial is
\[
P_{G_g} = x^{4g-4} - x^{4g-5} - x^{2g-1} - 10x^{2g-2} - x^{2g-3} - x + 1,
\]
and one may estimate $\lambda(\phi_g)$ by noting that it equals the maximum modulus of the roots of $P_{G_g}$, which is estimable due to the special form of $P_{G_g}$.  
Though more involved, this argument yields the better upper bound of 
\[
\log(\lambda(\phi_g)) \leqslant \frac{3\log(4g-4)}{(4g-4)}.
\]
\subsection{The main theorem.}

We can now assemble the proof of the main theorem.
\begin{proof}[Proof of Theorem \ref{T:main}]
Proposition \ref{P:upper bound} implies that
\[ 
	\kappa_g \leqslant \frac{2}{\log(g-\frac{1}{2})} \asymp \frac{1}{\log(g)}.
\]
Lemma \ref{L:comparing asymptotics} applied to the sequence $\phi_g:S_g \to S_g$ above, together with Proposition \ref{P:cc calc} and the upper bound in
Lemma \ref{LogGrowth-lem}, implies
\[ 
	\kappa_g \geqslant \frac{\ell_{\cC}(\phi_g)}{\log(\lambda(\phi_g))} \geqslant \frac{1/(2g-1)}{\log(10g-21)/(g-2)} \asymp \frac{1}{\log(g)}. \qedhere
\]
\end{proof}

\subsection{Extremal length.}
Masur and Minsky \cite{Mas-Min} use extremal length rather than hyperbolic length to define the map $\T(S) \to \curvskel{1}(S)$.  
Recall that the extremal length of a curve $\alpha$ with respect to $X$ in $\T(S)$ is $\Ext_X(\alpha) = 1/\mod_X(\alpha)$, where $\mod_X(\alpha)$ is the supremum of conformal moduli for embedded annuli with core curves homotopic to $\alpha$.  The set of curves with smallest extremal length,
\[ 
\short_\Ext(X) = \{ \alpha \mathrm{\ in\ } \curvskel{1}(S) \mid \Ext_X(\alpha) \leqslant\Ext_X(\beta)  \mbox{\ for all\ } \beta \in \curvskel{0}(S) \}, 
\]
is finite.    As with hyperbolic length, the set $\short_\Ext(X)$ has diameter bounded above by a constant $c = c(S)$ (Lemma 2.4 of \cite{Mas-Min}), and again we view $\short_\Ext$ as a map $\T(S) \to \curvskel{1}(S)$.  This map is also coarsely Lipschitz, and we let $\kappa_g^\Ext$ denote the optimal Lipschitz constant for $\short_\Ext:\T(S_g) \to \curvskel{1}(S_g)$.
\begin{proposition}\label{Extremal.Length.Corollary}
We have $\kappa_g = \kappa_g^\Ext$ for all $g$.  
In particular, $\kappa_g^\Ext \asymp \frac{1}{\log(g)}$.
\end{proposition}
\begin{proof}
Suppose $\alpha$ in $\short(X)$.  
The collar neighborhood of width $\ell_X(\alpha)/2$ from Lemma \ref{L:collar on short} provides a conformal annulus of definite modulus (depending on $\ell_X(\alpha)$), and hence $\Ext_X(\alpha) < L'$ for some $L' = L'(S)$.  Now let $\beta$  lie in $\short_\Ext(X)$, so that $\Ext_X(\beta) \leqslant L'$.  
By Lemma 2.5 of \cite{Mas-Min},
$
	d(\alpha,\beta) \leqslant  2 L' + 1.
$
From this we deduce
\[ 
	|\short(X) - \short_\Ext(X) | < 2L' + 1. 
\]
Therefore, if one of $\short$ or $\short_\Ext$ is $(K,C)$--coarsely Lipschitz, then, by the triangle inequality, the other is $(K,C+ 2(2L' + 1))$--coarsely Lipschitz.  The proposition follows.  
\end{proof}

\end{document}